\documentclass[11pt,a4paper]{article}

\usepackage[headinclude]{typearea}
\usepackage[english]{babel}           
\usepackage[T1]{fontenc}
\usepackage{scrtime}
\usepackage{amsmath,amsthm,amsfonts,amssymb}
\usepackage {color}                                  
\usepackage {pifont}                                 
\usepackage {multicol,multirow}                               
\usepackage{bbm}
\usepackage{comment}				
\usepackage[numbers]{natbib}                                                     
\usepackage[normalem]{ulem}                                                     
\usepackage [scaled=.90]{helvet}                     
\usepackage {courier}                                
\usepackage {graphicx}                               
\usepackage {array}                                  
\usepackage {longtable}                             
\usepackage{fancyvrb}
\usepackage{dsfont}
\usepackage{enumerate} 
\usepackage{hyperref}
\usepackage{ifpdf}
\usepackage{caption}
\usepackage{mathrsfs}

\usepackage{setspace}

\usepackage{marvosym}

\usepackage{mathtools}
\usepackage{bbold}
\mathtoolsset{showonlyrefs} 

\newtheorem{theorem}{Theorem}[section]

\newtheorem{lemma}[theorem]{Lemma}

\newtheorem{proposition}[theorem]{Proposition}

\theoremstyle{definition}

\newcommand{\exclude}[1]{}

\definecolor{darkgreen}{rgb}{0,0.5,0}
\definecolor{lightgreen}{rgb}{0.5,0.9,0.5}
\definecolor{magenta}{rgb}{0.75,0,0.25}

\definecolor{violet}{rgb}{0.25,0,0.75}

\title{Convergence of the tamed-Euler--Maruyama method for SDEs with discontinuous and polynomially growing drift}

\author{Kathrin Spendier \and Michaela Sz\"olgyenyi\thanks{K. Spendier and M. Sz\"olgyenyi are supported by the Austrian Science Fund (FWF): DOC 78.}} 
\date{Preprint, \today}

\allowdisplaybreaks
\begin{document}

\maketitle

\begin{abstract}
Numerical methods for SDEs with irregular coefficients are intensively studied in the literature, with different types of irregularities usually being attacked separately. In this paper we combine two different types of irregularities:~polynomially growing drift coefficients and discontinuous drift coefficients.
For SDEs that suffer from both irregularities we prove strong convergence of order $1/2$ of the tamed-Euler--Maruyama scheme from \cite{SpSz:HuJeKl2012}.\\

\end{abstract}

\section{Introduction}
We consider time-homogeneous stochastic differential equations (SDEs),
\begin{align}\label{SpSz:SDE}
dX_t &= \mu(X_t)  dt + \sigma (X_t) dW_t, \quad t\in[0,T], \quad X_0=\xi,
 \end{align}
where $\xi\in\mathbb{R}$, $\mu,\sigma\colon\mathbb{R} \to \mathbb{R}$ are measurable functions, $T\in(0,\infty)$, $W=(W_t)_{t\in[0,T]}$ is a
standard Brownian motion on the filtered probability space $(\Omega,\mathcal{F},\mathbb{P},(\mathbb{F}_t)_{t\in[0,T]})$, where the filtration satisfies the usual conditions. 

The novelty in this paper is that we allow $\mu$ to be piecewise Lipschitz and polynomially growing at the same time. Note that the explicit Euler--Maruyama approximation fails to satisfy moment bounds for SDEs with superlinearly growing coefficients. 
Therefore, the following numerical method from \cite{SpSz:HuJeKl2012} is used for approximating the solution of SDE \eqref{SpSz:SDE}. Let $N\in\mathbb{N}$ and define the equidistant time grid $0=t_0<t_1<\dots<t_N=T$ with $t_{k+1}-t_k=\delta$ for all $k\in\{0,\dots,N-1\}$. 
Denote for all $t\in[0,T]$, ${\underline{t}}:=\max\{ t_k\colon t\ge t_k \}$.
The time-continuous tamed-Euler--Maruyama (EM) scheme is given by $X^{(\delta)}_0=\xi$ and
\begin{align} \label{SpSz:tamedeulerapprox}
X^{(\delta)}_{t} = X^{(\delta)}_{\underline{t}}+ \frac{ \mu(X^{(\delta)}_{\underline{t}})(t-{\underline{t}}) }{1+\delta |\mu(X^{(\delta)}_{\underline{t}})|}+ \sigma(X^{(\delta)}_{\underline{t}})(W_{t}-W_{{\underline{t}}}), \quad t\in[0,T].
\end{align}     
In this method the drift part $\mu(X^{(\delta)}_{\underline{t}})\delta$ is ``tamed'' by the factor $1/(1+\delta |\mu(X^{(\delta)}_{\underline{t}})|)$, which makes it bounded by $1$, preventing the drift from becoming extraordinarily large.

In case the coefficients $\mu$ and $\sigma$ are Lipschitz, it is well known that SDE \eqref{SpSz:SDE} admits a unique strong solution which can be approximated with the EM scheme at strong convergence order $1/2$. In this paper we consider two different types of irregularities of the drift coefficient that have been studied separately in the literature.

For existence and uniqueness results of solutions to SDEs with discontinuous drift coefficient, see  \cite{SpSz:Zv,SpSz:Ve1981,SpSz:Ve1984,SpSz:LeSzTh,SpSz:LeSz2016,SpSz:ShSz,SpSz:LeSz2017,SpSz:PrSz,SpSz:PrSzXu}.
Approximation results for SDEs with discontinuous drift are studied, e.g.,~in 
\cite{SpSz:NgTa2016,SpSz:LeSz2016,SpSz:NgTa2017,SpSz:LeSz2017,SpSz:LeSz2017b,SpSz:LeSz2018,SpSz:NeSzSz,SpSz:DaGe,SpSz:MGYa2020,
SpSz:DaGeLe,SpSz:NeSz,SpSz:Ya,SpSz:PrSz,SpSz:Sz,SpSz:MGYa2022,SpSz:PrScSz}.
For lower error bounds for approximation schemes based on finitely many evaluations of the driving Brownian motion, see \cite{SpSz:HeHeMG,SpSz:MGYa2021}. 
In this paper we focus on SDEs that locally around the points of discontinuity satisfy a piecewise Lipschitz condition. Convergence rates for SDEs with piecewise Lipschitz drift have first been studied in \cite{SpSz:LeSz2016,SpSz:LeSz2017,SpSz:LeSz2018}. There, a transformation method and a proof technique is introduced that works by removing the discontinuities from the drift. This idea has been taken up and advanced in \cite{SpSz:MGYa2020}, who prove order $1/2$ convergence of the crude EM scheme in the same setting. Further contributions in this spirit are \cite{SpSz:LeSz2017b,SpSz:LeSz2018,SpSz:NeSzSz,SpSz:Ya,SpSz:PrSz,SpSz:Sz,SpSz:MGYa2022,SpSz:PrScSz, SpSz:MGRaYa}. 
For SDEs with superlinearly growing coefficients, the explicit EM scheme does not converge in the strong mean square sense, see \cite{SpSz:HuJeKl2011}. 
Under local boundedness, one-sided Lipschitz and polynomially growing drift, and non-degeneracy of the diffusion plus technical conditions, pathwise convergence is shown in \cite{SpSz:Gy}. 
The implicit EM scheme converges in the same setting as in \cite{SpSz:HuJeKl2012}, that is under a continuously differentiable one-sided Lipschitz drift with polynomially growing derivative and a globally Lipschitz continuous diffusion.
In \cite{SpSz:HuJeKl2012} the explicit tamed-EM scheme is introduced and proven to converge at strong order $1/2$. A similar taming technique is introduced in \cite{SpSz:Sa2013}. For further results, see \cite{SpSz:KuSa2014,SpSz:HuJe,SpSz:Sa2016,SpSz:Ma,SpSz:Do,SpSz:KuSa2019,SpSz:GaHeWa,SpSz:Ku}.

In the current paper we prove strong convergence of order $1/2$ of the tamed-EM scheme from \cite{SpSz:HuJeKl2012} for scalar SDEs with a drift coefficient that has finitely many discontinuities while growing superlinearly. We reach our goal by combining ideas from \cite{SpSz:HuJeKl2012,SpSz:LeSz2016,SpSz:LeSz2018,SpSz:MGYa2020}.

At this point, note that parallel and independent work has been done in \cite{SpSz:MGSaYa}, where existence, uniqueness, and $L_p$-approximation results for a similar tamed-EM scheme is presented for SDEs for locally piecewise Lipschitz and polynomially growing drift. The assumptions on the coefficients in \cite{SpSz:MGSaYa} are more general since not only the drift, but also the diffusion may grow polynomially. Note that our article is cited in \cite{SpSz:MGSaYa} as well. \cite{SpSz:MGSaYa} was on arXiv earlier than our paper. Both, our result and the one from \cite{SpSz:MGSaYa} have been presented at conferences already in 2021. Also in parallel and independent, \cite{SpSz:HuGa} has proven existence, uniqueness, and $L_p$-convergence of order $1/2$ of the tamed-EM scheme from \cite{SpSz:HuJeKl2012}.

\section{Convergence result}
\textbf{Assumption 1} \\
We assume the following on the coefficients of SDE \eqref{SpSz:SDE}. 
\begin{itemize}
\item [(i-1)]  The drift coefficient $\mu\colon\mathbb{R} \to \mathbb{R}$ is piecewise Lipschitz on $[\zeta_1,\zeta_m]$ with $m\in\mathbb{N}$ discontinuities in the points $\zeta_1 < \hdots < \zeta_m\in\mathbb{R}$.

\item [(i-2)]  The drift coefficient $\mu\colon\mathbb{R} \to \mathbb{R}$ is globally one-sided Lipschitz on $(-\infty, \zeta_1)$ and $(\zeta_m, \infty)$, i.e there exists a constant $L_{\mu} \in (0,\infty)$ such that for all $x,y \in (-\infty, \zeta_1)$ or $(\zeta_m,\infty)$, $(x-y)(\mu(x)-\mu(y)) \leq L_{\mu}|x-y|^2$, and $\mu \in 
C^1$ on each of the intervals $(-\infty, \zeta_1), (\zeta_1,\zeta_2),...,(\zeta_m,\infty)$ with, wherever it exists, at most polynomially growing derivative.

\item [(ii)]  The diffusion coefficient $\sigma\colon\mathbb{R}\to\mathbb{R}$ is Lipschitz continuous, lies in $\mathcal{C}^1$, $\sigma'$ is Lipschitz, and for all $k\in\{1,\dots,m\}$, $\sigma(\zeta_k) \neq 0$.
\end{itemize}
In the following, for any one-sided or globally Lipschitz function $f$ we denote its one-sided or global Lipschitz constant by $L_f$. 
Note that by \cite[Lemma 2.2]{SpSz:PrSz}, piecewise Lipschitz functions grow at most linearly. We denote for any function that grows at most linearly by $c_f$ the smallest constant such that $|f(x)|\le c_f(1+|x|)$. 
\begin{theorem}\label{SpSz:mainmain}
Let Assumption 1 hold.
Then there exists $C^{(\text{tEM})}\in(0,\infty)$ such that for all $\delta \in(0,1)$ sufficiently small,
\begin{equation*} 
\bigg(\mathbb{E} \bigg[\sup_{t\in[0,T]}|X_t- X^{(\delta)}_t|^2\bigg]\bigg)^{\!1/2}\le C^{(\text{tEM})}\delta^{1/2}.
\end{equation*}
\end{theorem}
\subsection{Structure of the proof} \label{SpSz:first_part}
Our convergence proof is based on a transformation trick from \cite{SpSz:LeSz2016,SpSz:LeSz2018} in combination with ideas from \cite{SpSz:MGYa2020} and \cite{SpSz:HuJeKl2012}.
Let $Z = G(X)$, where $G$ is a specific transformation, which we are going to introduce in Section \ref{SpSz:transform}. Due to the fact that $G$ will be invertible and bi-Lipschitz, we have $X = G^{-1}(Z)$ and get that
\begin{equation} \label{SpSz:est-lip}
\begin{aligned}
\bigg( \mathbb{E}\bigg[\sup_{t\in[0,T]}|X_t- X^{(\delta)}_t|^2\bigg]\bigg)^{\!1/2} \le
L_{G^{-1}}  \bigg(\mathbb{E}\bigg[\sup_{t\in[0,T]}|Z_t- G(X^{(\delta)}_t)|^2\bigg]\bigg)^{\!1/2}.
\end{aligned}
\end{equation}
By the triangle inequality,
\begin{equation}\label{SpSz:est-triangle1}
\begin{aligned}
 \bigg(\mathbb{E}\bigg[\sup_{t\in[0,T]}|Z_t- G(X^{(\delta)}_t)|^2\bigg]\bigg)^{\!1/2}
&\le
 \bigg(\mathbb{E}\bigg[\sup_{t\in[0,T]}|Z_t- Z^{(\delta)}_t|^2\bigg]\bigg)^{\!1/2}
 \\&\quad+
 \bigg(\mathbb{E}\bigg[\sup_{t\in[0,T]}|Z^{(\delta)}_t- G(X^{(\delta)}_t)|^2\bigg]\bigg)^{\!1/2},
 \end{aligned}
\end{equation}
where $Z^{(\delta)}$ is defined as in \eqref{SpSz:tamedeulerapprox} with $\mu, \sigma$ replaced by $\tilde{\mu}, \tilde{\sigma}$. For estimating the first error term in \eqref{SpSz:est-triangle1}, we basically apply the result from \cite{SpSz:HuJeKl2012}.
We need to show that the tamed-EM approximation $Z^{(\delta)}$ converges to $Z$ with strong rate $1/2$. 
For estimating the second error term in \eqref{SpSz:est-triangle1}, we extend ideas from \cite{SpSz:MGYa2020} for the estimation of discontinuity crossing probabilities of the tamed-EM scheme and do a case by case analysis. 
\subsection{Properties of the transformed SDE} \label{SpSz:transform}
All the material we use in this subsection is presented in \cite{SpSz:LeSz2017,SpSz:MGYa2022}. We repeat it for the convenience of the reader.
We will apply a transform $G\colon \mathbb{R}\to \mathbb{R}$ from \cite{SpSz:MGYa2022} that has the property that the process $Z \colon \Omega \times [0,T]  \rightarrow \mathbb{R}$ formally defined by
$Z_t = G(X_t), ~ t \in [0,T]$,
satisfies the SDE 
\begin{equation} \label{SpSz::Z1}
\begin{aligned}
dZ_t = \tilde{\mu}(Z_t) \,dt + \tilde{\sigma}(Z_t) \,dW_t, \qquad Z_0 = G(\xi), \nonumber
\end{aligned}
\end{equation}
with $\tilde{\mu}, \tilde{\sigma}$ given by
\begin{equation}\label{SpSz:tildecoeff}
\widetilde \mu=\left(G'\cdot \mu +\tfrac{1}{2}G''\cdot\sigma^2\right)\circ G^{-1} \, \text{ and }\, \widetilde\sigma=(G'\cdot\sigma)\circ G^{-1}.
\end{equation}
The transform $G$ is chosen in a way so that $\tilde{\mu}, \tilde{\sigma}$ are sufficiently well-behaved and so that it impacts the coefficients of the SDE \eqref{SpSz:SDE} only locally around the points of discontinuity of the drift.  
The specific form of $G$ is a non-unique choice that has the following properties.
\begin{lemma}  \cite[Lemma 1]{SpSz:MGYa2022}\label{SpSz:propertiesofG} ~ Let Assumption 1 hold. Let $G$ be as in \cite{SpSz:MGYa2022}. Then
\begin{itemize}
\item [(i)]  $G$ is differentiable on $\mathbb{R}$ with Lipschitz continuous and bounded derivative $G'$ that satisfies $G'(\zeta_k) = 1$ for all $k \in \{1,...,m \}$. Furthermore, there exists $\varepsilon \in (0, \infty)$ such that for every $x \in \mathbb{R}$ with $|x| > \varepsilon, ~ G'(x) =1$.
\item [(ii)] $G$ has an inverse $G^{-1} \colon \mathbb{R} \rightarrow \mathbb{R}$ that is Lipschitz continuous.
\item [(iii)] The function $G'$ is twice differentiable on each of the intervals $(-\infty, \zeta_1), (\zeta_1,\zeta_2),\\ ~~ ~...,(\zeta_m,\infty)$ with Lipschitz continuous and bounded derivatives $G''$ and $G'''$.
\end{itemize}
\end{lemma}
\begin{lemma}\label{SpSz:tildecoefficients}
Let Assumption 1 hold.
Then $\tilde \mu$ and $\tilde\sigma$ from \eqref{SpSz:tildecoeff} satisfy the following:
\begin{itemize}
\item[(i)] The drift coefficient $\tilde{\mu}\colon\mathbb{R} \to \mathbb{R}$ is Lipschitz continuous on $[\zeta_1,\zeta_m]$.
\item[(ii)] The drift coefficient $\tilde{\mu}\colon\mathbb{R} \to \mathbb{R}$ is globally one-sided Lipschitz, in $C^1$ on each of the intervals $(-\infty, \zeta_1), (\zeta_1,\zeta_2),...,(\zeta_m,\infty)$, and has, wherever it exists, an at most polynomially growing derivative.
\item[(iii)] The diffusion coefficient $\tilde{\sigma}\colon\mathbb{R}\to\mathbb{R}$ is Lipschitz continuous on $\mathbb{R}$, and for all $k\in\{1,\dots,m\}$, $\tilde{\sigma}(\zeta_k) \neq 0$.
\end{itemize} 
\end{lemma} 
\begin{proof}
\begin{itemize}
\item [(i)] is proven in \cite[Lemma 2.4.]{SpSz:LeSz2017}.
\item [(ii)] The proof that $\tilde{\mu} \in C^1$ on $\bigcup_{k=1}^{m-1}(\zeta_k,\zeta_{k+1})$ can be found in \text{\cite[Lemma 2]{SpSz:MGYa2022}}.
Due to (i) it holds that $\tilde{\mu}$ is Lipschitz continuous on $[\zeta_1,\zeta_{m}]$ and therefore linearly growing.
It holds by Assumption 1 that there exists $\epsilon \in (0,\infty)$ such that on $(\zeta_1 - \epsilon, \zeta_m + \epsilon)^c$, $\tilde{\mu} = \mu$, hence $\tilde{\mu}$ has a polynomially growing derivative on $(\zeta_1 - \epsilon, \zeta_m + \epsilon)^c$. On $(\zeta_1 - \epsilon, \zeta_1)$ and $(\zeta_m,\zeta_m + \epsilon)$ it holds that $\tilde{\mu}$ is locally Lipschitz on a bounded domain, therefore Lipschitz, which implies also the polynomial growth condition. 
Finally, the one-sided Lipschitz property of $\tilde{\mu}$ follows from the fact that $\tilde{\mu}$ is Lipschitz on $(\zeta_1-\epsilon, \zeta_m + \epsilon)$ and $\tilde{\mu} = \mu$ on $(\zeta_1-\epsilon, \zeta_m + \epsilon)^c$.
\item [(iii)] For the proof of (iii), see \cite[Lemma 2.5.]{SpSz:LeSz2017}.
\end{itemize} 
\end{proof}
\begin{lemma} \label{SpSz:help1}
Let Assumption 1 hold. Then there exists a constant $\tilde{c} \in (0,\infty)$ such that for all $x,y \in \mathbb{R}$,
\begin{equation}\label{SpSz:Lcondition}
\begin{aligned} 
&|\tilde{\mu}(x)-\tilde{\mu}(y)| \leq \tilde{c}(1+|x|^{\tilde{c}} + |y|^{\tilde{c}})|x-y|. 
\end{aligned}
\end{equation}
Furthermore, there exists $c \in (0,\infty)$ such that for all $x \in \mathbb{R}$,
\begin{align*}
|\mu(x)| \leq c(1+|x|^{\tilde{c}+1}).
\end{align*}
\end{lemma}
\begin{proof}
There exists $\epsilon \in (0,\infty)$ such that the function $\tilde{\mu}$ lies in $C^1$ on $(-\infty, \zeta_1 - \epsilon)$ and $(\zeta_m + \epsilon, \infty)$ and is globally Lipschitz on $[\zeta_1 - \epsilon, \zeta_m + \epsilon]$. Hence, the fundamental theorem of calculus holds for $\tilde{\mu}$ on the whole of $\mathbb{R}$. By the polynomial growth of $\tilde{\mu'}$ there exists $\tilde{c} \in (0,\infty)$ such that for all $x,y \in \mathbb{R}$, $x \geq y$ (w.l.o.g.),
\begin{equation*}
\begin{aligned}
& |\tilde{\mu}(x)-\tilde{\mu}(y)| = \left| \int_{y}^{x} \tilde{\mu}'(t) \, dt \right| \leq \int_{y}^{x} |\tilde{\mu}'(t)| \, dt  \leq  \int_{y}^{x} \tilde{c}(1+ |t|^{\tilde{c}}) \, dt.
\end{aligned}
\end{equation*}
Solving the integral yields \eqref{SpSz:Lcondition}. For the second claim of the lemma, note that
\begin{align*} 
|\mu(x)| \leq |\mu(x)|\mathbb{1}_{(\zeta_1 - \epsilon, \zeta_m + \epsilon)}(x) + |\mu(x)|\mathbb{1}_{(- \infty, \zeta_1 - \epsilon]}(x) + |\mu(x)|\mathbb{1}_{[\zeta_m + \epsilon, \infty)}(x).
\end{align*}
On $(\zeta_1 - \epsilon, \zeta_m + \epsilon)$ we may apply the linear growth condition on $\mu$ and for the other two terms, choose $\epsilon$ so that $\mu = \tilde{\mu}$ and apply \eqref{SpSz:Lcondition}. This gives that there exists $c \in (0, \infty)$ such that
\begin{align*}
|\mu(x)| \leq  c(1+ |x|^{\tilde{c}+1}).
\end{align*}
\end{proof}
\begin{theorem}
Let Assumption 1 hold.
Then there exists a constant $c \in \left(0, \infty \right) $ such that
\begin{align}
\label{SpSz:Mainres1}
\mathbb{E}\left[\sup_{t\in[0,T]}|Z_t- Z^{(\delta)}_t|^2\right]\le c \delta. 
\end{align}
\end{theorem}
This result follows in principle from \cite[Theorem 1.1]{SpSz:HuJeKl2012}. The only assumption that is not satisfied in our case is that $\tilde{\mu}'$ only exists on each of the intervals $(-\infty,\zeta_{1}), (\zeta_1, \zeta_2),..., (\zeta_m,\infty)$ and not on the whole of $\mathbb{R}$. In the proof of \cite[Theorem 1.1]{SpSz:HuJeKl2012} this is only used to prove Lemma \ref{SpSz:help1}, \eqref{SpSz:Lcondition} above, which we also obtain in our setting. Hence, \eqref{SpSz:Mainres1} holds.
\subsection{Preparatory lemmas}\label{SpSz:secondpart}
We present several lemmas, which we need for the proof of our main result.
The following lemma originates from \cite[Lemma 5]{SpSz:MGSaYa} and is adapted to our situation.
\begin{lemma} \label{SpSz:helpingstuff}
Let Assumption 1 hold. There exists a constant $c \in (0,\infty)$ such that for sufficiently small $\delta \in (0,1)$ and for all $x\in\mathbb{R}$ it holds that
\begin{align*}
\frac{|\mu(x)|}{1+\delta|\mu(x)|} \leq c \cdot \min(\frac{1}{\sqrt{\delta}}(1+|x|), |\mu(x)|).
\end{align*}
\end{lemma}
\begin{lemma} \label{SpSz:lemmaonS}    Let Assumption 1 hold and let $p\in[2,\infty)$. Then there exists a constant $C_p^{(\text{M})}\in(0,\infty)$   such that for all $\delta\in(0,1)$ sufficiently small,
$$ \mathbb{E} \Big[ \sup_{t \in [0,T]} |X^{(\delta)}_t  |^p   \Big] \leq  C_p^{(\text{M})}    $$  and such that for all $\delta \in (0,1)$ sufficiently small, $t \in[0,T]$,
\begin{align*}
\left(\mathbb{E} [\sup_{s \in [t,t+\delta]} |X^{(\delta)}_s - X^{(\delta)}_{t} |^p ] \right)^{1/p} &\leq C_p^{(\text{M})} \cdot \delta^{1/2}.
\end{align*}
\end{lemma}
\begin{proof}
For proving the first statement, we would like to apply \cite[Lemma 3.9]{SpSz:HuJeKl2012}. However, since the assumptions in \cite{SpSz:HuJeKl2012} are different from ours in the sense that in \cite{SpSz:HuJeKl2012} they assume $\mu$ to be in $C^1$, we need to reprove \cite[Lemma 3.1 -- Lemma 3.9]{SpSz:HuJeKl2012}. The crucial estimates are (36), (39), and (40) from \cite[Lemma 3.1]{SpSz:HuJeKl2012}. We begin with proving these estimates.
To this end, choose $\epsilon \in (0, \infty)$ so that on $(-\infty, \zeta_1 - \epsilon) \cup (\zeta_m + \epsilon, \infty), \mu = \tilde{\mu}$, and let $\lambda := \big( 3 \cdot \big(1 + c + \tilde{c} + c_{\mu} + L_{\sigma} + L_{\mu} + T + |\mu(0)| + |\sigma(0)| + |\mu(1)| + |\mu(-1)| + |\mu(\zeta_1 - \epsilon)| + |\mu(\zeta_m + \epsilon)|\big) \big)^4 $, where $c, \tilde{c}$ are as in Lemma \ref{SpSz:help1}.
Furthermore, let $D_{t_{k}}^{(\delta)}, \Omega_{t_{k}}^{(\delta)}$ be defined as in \cite[(13), (14)]{SpSz:HuJeKl2012}, but replacing $c \in (0,\infty)$ in \cite[(14)]{SpSz:HuJeKl2012} by $k := \max\{c, \tilde{c}, c_{\mu}, L_{\mu}, L_{\sigma} \}$ for our setup.
We start with proving \cite[(36)]{SpSz:HuJeKl2012} in our setting. First of all, 
 on $\Omega_{t_{k+1}}^{(\delta)} \cap \{ \omega \in \Omega \colon |X_{t_{k}}^{(\delta)}| \leq 1 \}$, 
\begin{equation*}
|X_{t_{k+1}}^{(\delta)}| \leq 1+ T\cdot |\mu(X_{t_{k}}^{(\delta)})| + L_{\sigma} + |\sigma(0)|.
\end{equation*}
On $\Omega_{t_{k+1}}^{(\delta)} \cap \{ \omega \in \Omega \colon |X_{t_{k}}^{(\delta)}| \leq 1 \}$, for all $k \in \{0,...,N-1\}$ and all $\delta \in (0,1)$ sufficiently small, we get by Lemma \ref{SpSz:help1},
$$|\mu(X_{t_{k}}^{(\delta)})|  \leq c(1 + |X_{t_{k}}^{(\delta)}|^{\tilde{c}+1}) \leq 2c.$$
Therefore,
$$|X_{t_{k+1}}^{(\delta)}| \leq 1+ T\cdot2c + L_{\sigma} + |\sigma(0)| \leq \lambda,$$
on $\Omega_{t_{k+1}}^{(\delta)} \cap \{ \omega \in \Omega \colon |X_{t_{k}}^{(\delta)}| \leq 1 \}$, for all $k \in \{0,...,N-1\}$ and all $\delta \in (0,1)$ sufficiently small.
To prove \cite[(39)]{SpSz:HuJeKl2012}, observe that for all $x \in \mathbb{R}$ with $|x| \geq 1 $,
\begin{align} \label{SpSz:1}
x \mu(x) &= x \mu(x) \mathbb{1}_{(- \infty, \zeta_1 - \epsilon]}(x)  + x  \mu(x) \mathbb{1}_{(\zeta_1 - \epsilon , \zeta_m + \epsilon)}(x) + x \mu(x) \mathbb{1}_{[ \zeta_m + \epsilon, \infty)}(x) .
\end{align}
For $x \in (\zeta_1 - \epsilon, \zeta_m + \epsilon)$ with $|x| \geq 1$ it holds that
\begin{align}  \label{SpSz:2}
x \mu(x) \leq |x| \cdot |\mu(x)| \leq |x| c_{\mu}(1 + |x|) \leq \sqrt{\lambda} |x|^2.
\end{align}
For $x \in (- \infty, \zeta_1 - \epsilon]$ with $|x| \geq 1$ it holds that
\begin{equation}  \label{SpSz:3}
\begin{aligned}
x \mu(x) \mathbb{1}_{(- \infty, \zeta_1 - \epsilon]}(x) &= x \mu(x) \mathbb{1}_{(- \infty, \zeta_1 - \epsilon]}(x)  \mathbb{1}_{(- \infty, -1]}(\zeta_1 - \epsilon) \\
\quad & +  x \mu(x) \mathbb{1}_{(- \infty, \zeta_1 - \epsilon]}(x)  \mathbb{1}_{(-1, 1)}(\zeta_1 - \epsilon) \\
\quad & +  x \mu(x) \mathbb{1}_{(- \infty, \zeta_1 - \epsilon]}(x)  \mathbb{1}_{[1, \infty)}(\zeta_1 - \epsilon) .
\end{aligned}
\end{equation}  
By the one-sided Lipschitz condition on $\mu$ it holds for $\zeta_1 - \epsilon \leq -1$ and $x \in (- \infty, \zeta_1 - \epsilon]$,
\begin{equation}  \label{SpSz:4}
\begin{aligned}
x \mu(x) &\leq  x (\mu(x) - \mu(\zeta_1 - \epsilon))  + x \mu(\zeta_1 - \epsilon) \\
 \quad & \leq  L_{\mu} |x|^2 + |x|^2  |\mu(\zeta_1 - \epsilon)| \leq \sqrt{\lambda}|x|^2. 
\end{aligned}
\end{equation}
Here we used that by the one-sided Lipschitz condition,
\begin{equation*}
\mu(x) - \mu(\zeta_1 - \epsilon) \geq \frac{L_{\mu} | x- (\zeta_1 - \epsilon)|^2}{(x-(\zeta_1 - \epsilon))}
\end{equation*}
and hence  $x(\mu(x) - \mu(\zeta_1 - \epsilon)) \leq |x| \cdot L_{\mu} | x- (\zeta_1 - \epsilon)| \leq L_{\mu}|x|^2$.
For the other parts of \eqref{SpSz:3} we proceed analogously to \eqref{SpSz:4} inserting $\mu(-1)$  respectively $\mu(1)$ instead of $\mu(\zeta_1-\epsilon)$ to obtain for all $x \in \mathbb{R}$ with $|x| \geq 1$,
\begin{equation} \label{SpSz:5}
x \mu(x) \mathbb{1}_{(- \infty, \zeta_1 - \epsilon]}(x) \leq \sqrt{\lambda} |x|^2.
\end{equation}
For the term $x \mu(x) \mathbb{1}_{[\zeta_m + \epsilon, \infty)}(x)$ we proceed analogously to obtain 
\begin{equation} \label{SpSz:6}
x \mu(x) \mathbb{1}_{[ \zeta_m + \epsilon, \infty)}(x) \leq \sqrt{\lambda} |x|^2.
\end{equation}
Putting \eqref{SpSz:1}, \eqref{SpSz:2}, \eqref{SpSz:5}, and \eqref{SpSz:6} together yields \cite[(39)]{SpSz:HuJeKl2012}.\\
To show \cite[(40)]{SpSz:HuJeKl2012}, let $k := \max\{c, \tilde{c}, c_{\mu}, L_{\mu}, L_{\sigma} \}$, and consider for all $N \in \mathbb{N}$, $x \in \mathbb{R}$ with $|x| \in [1, N^{1/(2k)}]$. We write
\begin{align} \label{SpSz:7}
| \mu(x) | &= |\mu(x) | \mathbb{1}_{(- \infty, \zeta_1 - \epsilon]}(x)  + | \mu(x)| \mathbb{1}_{(\zeta_1 - \epsilon , \zeta_m + \epsilon)}(x) + | \mu(x)| \mathbb{1}_{[ \zeta_m + \epsilon, \infty)}(x) .
\end{align}
For the first part of \eqref{SpSz:7} we write
\begin{equation}  \label{SpSz:8}
\begin{aligned}
|\mu(x)| \mathbb{1}_{(- \infty, \zeta_1 - \epsilon]}(x) &= |\mu(x) | \mathbb{1}_{(- \infty, \zeta_1 - \epsilon]}(x)  \mathbb{1}_{(- \infty, -N^{1/(2k)})}(\zeta_1 - \epsilon) \\
\quad & +  | \mu(x)| \mathbb{1}_{(- \infty, \zeta_1 - \epsilon]}(x)  \mathbb{1}_{[-N^{1/(2k)}, -1]}(\zeta_1 - \epsilon) \\
\quad & +  | \mu(x)| \mathbb{1}_{(- \infty, \zeta_1 - \epsilon]}(x)  \mathbb{1}_{(-1, 1)}(\zeta_1 - \epsilon) \\
\quad & +  | \mu(x)| \mathbb{1}_{(- \infty, \zeta_1 - \epsilon]}(x)  \mathbb{1}_{[1, \infty)}(\zeta_1 - \epsilon) .
\end{aligned}
\end{equation}  
The first summand of \eqref{SpSz:8} is $0$, as for $|x| \in [1, N^{1/(2k)}]$ it is impossible that both conditions are satisified simultaneaously. For the second summand we will make use of Lemma \ref{SpSz:help1} and get
\begin{equation} \label{SpSz:9}
\begin{aligned}
& |\mu(x)| \mathbb{1}_{(- \infty, \zeta_1 - \epsilon]}(x)  \mathbb{1}_{[-N^{1/(2k)}, -1]}(\zeta_1 - \epsilon)  \\
& \leq  (|\mu(x) - \mu(\zeta_1 - \epsilon)| + |\mu(\zeta_1 - \epsilon)|) \mathbb{1}_{(- \infty, \zeta_1 - \epsilon]}(x)  \mathbb{1}_{[-N^{1/(2k)}, -1]}(\zeta_1 - \epsilon)  \\
& \leq \big(\tilde{c}(1 + |x|^{\tilde{c}} + |\zeta_1 - \epsilon|^{\tilde{c}})|x-(\zeta_1 - \epsilon)|  \\
&\quad + |x||\mu(\zeta_1 - \epsilon)|\big)\mathbb{1}_{(- \infty, \zeta_1 - \epsilon]}(x)  \mathbb{1}_{[-N^{1/(2k)}, -1]}(\zeta_1 - \epsilon)  \\
& \leq (k(1 + 2 \sqrt{N})  + |\mu(\zeta_1 - \epsilon)|)|x|\mathbb{1}_{(- \infty, \zeta_1 - \epsilon]}(x)  \mathbb{1}_{[-N^{1/(2k)}, -1]}(\zeta_1 - \epsilon)  \\
& \leq (3k + |\mu(\zeta_1 - \epsilon)|) \sqrt{N}|x| \mathbb{1}_{(- \infty, \zeta_1 - \epsilon]}(x)  \mathbb{1}_{[-N^{1/(2k)}, -1]}(\zeta_1 - \epsilon)  \\
& \leq \sqrt[4]{\lambda} \sqrt{N} |x|  \mathbb{1}_{(- \infty, \zeta_1 - \epsilon]}(x)  \mathbb{1}_{[-N^{1/(2k)}, -1]}(\zeta_1 - \epsilon).
\end{aligned}
\end{equation}
Analogously, we proceed with the remaining two summands of \eqref{SpSz:9} using Lemma \ref{SpSz:help1} and inserting $\mu(-1)$  respectively $\mu(1)$ instead of $\mu(\zeta_1-\epsilon)$ to obtain 
\begin{equation} \label{SpSz:10}
| \mu(x) | \mathbb{1}_{(- \infty, \zeta_1 - \epsilon]}(x) \leq \sqrt[4]{\lambda} \sqrt{N} |x|  \mathbb{1}_{(- \infty, \zeta_1 - \epsilon]}(x) .
\end{equation}
Analogously, we proceed for $|\mu(x)| \mathbb{1}_{[ \zeta_m+ \epsilon, \infty)}(x)$ to obtain that
\begin{equation} \label{SpSz:11}
| \mu(x) | \mathbb{1}_{[ \zeta_m+ \epsilon, \infty)}(x) \leq \sqrt[4]{\lambda} \sqrt{N} |x|  \mathbb{1}_{[ \zeta_m+ \epsilon, \infty)}(x) .
\end{equation}
Combining \eqref{SpSz:7}, \eqref{SpSz:10}, \eqref{SpSz:11}, the linear growth of $\mu$ on $(\zeta_1 - \epsilon, \zeta_m + \epsilon)$ and squaring the obtained inequalities yields for all $x\in \mathbb{R}$, with $|x| \in [1, N^{1/(2k)}]$,
\begin{equation*}
| \mu(x) |^2 \leq \sqrt{\lambda}N |x|^2,
\end{equation*}
which proves \cite[(40)]{SpSz:HuJeKl2012}. 
After having proven the estimates (36), (39), and (40) from \cite[Proof of Lemma 3.1]{SpSz:HuJeKl2012}, we are able to prove \cite[Lemma 3.1 -- Lemma 3.9]{SpSz:HuJeKl2012} in the same way as in \cite{SpSz:HuJeKl2012}. Therefore, we may apply \cite[Lemma 3.9]{SpSz:HuJeKl2012} to our setting. By \cite[Lemma 3.9]{SpSz:HuJeKl2012} there exists $c_{1_{p}} \in(0,\infty)$ such that for all $\delta\in(0,1)$ sufficiently small,
\begin{equation} \label{SpSz:eq8}
\sup_{\delta \in (0,1)} \sup_{t \in [0,T]} \mathbb{E} [|X^{(\delta)}_{\underline{t}}  |^p   ] \leq  c_{1_{p}}.
\end{equation} 
We will use this and follow the steps from \cite[Proof of Lemma 3.1]{SpSz:Sa2013}. By Lemma \ref{SpSz:helpingstuff} and \eqref{SpSz:eq8} we get that there exist $c_2, c_3, c_4 \in (0,\infty)$ such that 
\begin{equation} \label{SpSz:eq9}
\begin{aligned}
&\sup_{t \in [0,T]} \mathbb{E} \big[ | X^{(\delta)}_t - X^{(\delta)}_{\underline{t}}|^{p}\big] \\
&\leq  \sup_{t \in [0,T]} 2^{p-1} \Bigg( \mathbb{E}\Bigg[ \Bigg|  \frac{\mu(X^{(\delta)}_{\underline{t}})}{1+ \delta |\mu(X^{(\delta)}_{\underline{t}})|} \Bigg|^p\Bigg] \cdot \delta^p + c_2 \cdot \mathbb{E} [|  \sigma(X^{(\delta)}_{\underline{t}})|^{p} ] \cdot \delta^{p/2} \Bigg) \\
&\leq  2^{p-1} (c_3^p + c_{\sigma}^p\cdot c_2) \cdot  \delta^{p/2} \cdot 
\left(1+ \sup_{\delta \in (0,1)} \sup_{t \in [0,T]} \mathbb{E} | X^{(\delta)}_{\underline{t}} |^p ] \right)  \leq c_4 \delta^{p/2}.
\end{aligned}
\end{equation}
By this, \eqref{SpSz:eq8}, H\"older's inequality, and Lemma \ref{SpSz:helpingstuff}, there exists $c_5 \in (0, \infty)$ such that
\begin{equation}\label{SpSz:driftproduct}
\begin{aligned}
& \mathbb{E} \Bigg[ |X^{(\delta)}_t - X^{(\delta)}_{\underline{t}}  |^p \cdot \Bigg|\frac{\mu(X^{(\delta)}_{\underline{t}})}{1+\delta|\mu(X^{(\delta)}_{\underline{t}})|} \Bigg|^p  \Bigg]  \leq c_4 \cdot \delta^{p/2} \cdot c_3^p \cdot \delta^{-p/2}  \cdot \Big( 1 +\\
&  \sup_{\delta \in (0,1)} \sup_{t \in [0,T]} \mathbb{E} |X^{(\delta)}_{\underline{t}}|^{2p}] \Big)^{1/2}  \leq c_4 \cdot \delta^{p/2} \cdot c_3^p \cdot \delta^{-p/2} \cdot (1 +  c_{1_{2p}})^{1/2} \leq c_5. 
\end{aligned}
\end{equation}
Next, It\^o's formula gives
\begin{equation} \label{SpSz:ito}
\begin{aligned}
&|X_t^{(\delta)}|^2  =|\xi|^2  + 2 \int_{0}^t X_{\underline{s}}^{(\delta)} \cdot \left(\frac{\mu(X_{\underline{s}}^{(\delta)})}{1+ \delta |\mu(X_{\underline{s}}^{(\delta)})|}\right) \,ds  + \int_{0}^t |\sigma(X_{\underline{s}}^{(\delta)})|^2 \,ds \\
& + 2 \int_{0}^t \left(X_{s}^{(\delta)} - X_{\underline{s}}^{(\delta)} \right) \cdot \left(\frac{\mu(X_{\underline{s}}^{(\delta)})}{1+ \delta |\mu(X_{\underline{s}}^{(\delta)})|}\right) \,ds  + 2 \int_{0}^t X_{s}^{(\delta)} \cdot \sigma(X_{\underline{s}}^{(\delta)}) \, dW_s.
\end{aligned}
\end{equation}
Furthermore, by using the linear growth of $\mu$ on $(\zeta_1 - \epsilon, \zeta_m + \epsilon)$ and the one-sided Lipschitz condition of $\mu$ on $(\zeta_1 - \epsilon, \zeta_m + \epsilon)^c$, it follows that there exists $c_6 \in (0,\infty)$ such that for all $x \in \mathbb{R}$,
\begin{equation} \label{SpSz:oslest}
\begin{aligned}
&\frac{2 x\mu(x)}{1 + \delta|\mu(x)|} \leq c_6(1 + |x|^2).
\end{aligned}
\end{equation}
By \eqref{SpSz:driftproduct}, \eqref{SpSz:ito}, \eqref{SpSz:oslest}, the linear growth of $\sigma$, the Burkholder-Davis-Gundy inequality, Young's inequality, and H\"older's inequality, for $p \geq 2$ there exist $c_7 \in (0, \infty)$, $c_8 \in (1,\infty)$  such that
\begin{align*} 
&\mathbb{E}  [ \sup_{s \in [0,t]}|X_s^{(\delta)}|^p] \leq  5^{\frac{p}{2}-1} \Bigg( \mathbb{E} [|\xi|^p] +  \mathbb{E}\left[ \sup_{s \in [0,t]} \left( \int_{0}^s | \sigma(X_{\underline{u}}^{(\delta)}) |^2 \,du \right)^{p/2} \right] \notag\\
&\quad + \mathbb{E}\left[ \sup_{s \in [0,t]} \left( \int_{0}^s 2 \cdot X_{\underline{u}}^{(\delta)} \cdot \left(\frac{\mu(X_{\underline{u}}^{(\delta)})}{1+ \delta |\mu(X_{\underline{u}}^{(\delta)})|}\right) \,du \right)^{p/2} \right] \notag \\
&\quad + \mathbb{E}\left[ \sup_{s \in [0,t]} \left(  \int_{0}^s 2 \cdot ( X_{u}^{(\delta)} - X_{\underline{u}}^{(\delta)} ) \cdot \left( \frac{\mu(X_{\underline{u}}^{(\delta)})}{1+ \delta |\mu(X_{\underline{u}}^{(\delta)})|}\right) \,du \right)^{p/2} \right] \notag \\
&\quad + \mathbb{E} \left[ \sup_{s \in [0,t]} \left| \int_{0}^s 2 \cdot X_{u}^{(\delta)} \cdot \sigma(X_{\underline{u}}^{(\delta)}) \, dW_u \right|^{p/2} \right] \Bigg) \notag\\
&\leq  5^{\frac{p}{2}-1} \Bigg( \mathbb{E} [|\xi|^p] + 
t^{\frac{p}{2}-1} \cdot \mathbb{E} \left[ \sup_{s \in [0,t]}  \int_{0}^s |\sigma(X_{\underline{u}}^{(\delta)}) |^{p} \,du \right] \\
&\quad  + t^{\frac{p}{2}-1} \cdot c_6^{p/2} \cdot 2^{\frac{p}{2}-1} \cdot \int_{0}^t \left( 1 + \mathbb{E} [|X_{\underline{s}}^{(\delta)}|^p] \right) \,ds  \notag \\
& \quad  + t^{\frac{p}{2}-1} \cdot 2^{\frac{p}{2}} \cdot \int_{0}^t \mathbb{E} \left[ | X_{u}^{(\delta)} - X_{\underline{u}}^{(\delta)} |^{\frac{p}{2}} \cdot \Bigg| \frac{\mu(X_{\underline{u}}^{(\delta)})}{1+ \delta |\mu(X_{\underline{u}}^{(\delta)})|}\Bigg|^{\frac{p}{2}} \right] \,du  \notag \\
&\quad  + 2^{\frac{p}{2}} \cdot c_7 \cdot \mathbb{E} \left[ \left( \int_{0}^t   | X_{s}^{(\delta)}|^2 \cdot |\sigma(X_{\underline{s}}^{(\delta)}) |^2 \, ds \right)^{p/4} \right] \Bigg) \notag\\
&  \leq  5^{\frac{p}{2}-1} \Bigg( \mathbb{E} [|\xi|^p] + 
t^{\frac{p}{2}-1} \cdot c_{\sigma}^p \cdot 2^{p-1} \cdot \int_{0}^t \left( 1 + \mathbb{E} [|X_{\underline{s}}^{(\delta)}|^p] \right) \,ds \notag \\
&\quad  + t^{\frac{p}{2}-1} \cdot c_6^{p/2} \cdot 2^{\frac{p}{2}-1} \cdot \int_{0}^t \left( 1 + \mathbb{E} [|X_{\underline{s}}^{(\delta)}|^p] \right) \,ds   + t^{\frac{p}{2}} \cdot 2^{\frac{p}{2}} \cdot c_5 \notag \\
&\quad + 2^{\frac{p}{2}} \cdot c_7 \cdot \mathbb{E} \left[ \left( \int_{0}^t   | X_{s}^{(\delta)}|^2 \cdot |\sigma(X_{\underline{s}}^{(\delta)}) |^2 \, ds \right)^{p/4} \right] \Bigg) \notag\\
&\leq  c_8 \Bigg( 1 + \mathbb{E} [|\xi|^p] + \int_{0}^t  \sup_{u \in [0,s]} \mathbb{E} [ |X_{\underline{u}}^{(\delta)}|^p] \,ds \notag\\
&\quad  +  \frac{1}{2 c_8}  \mathbb{E} [ \sup_{s \in [0,t]}|X_{s}^{(\delta)}|^p ] +  \frac{c_8}{2} \mathbb{E} \left[ \left( \int_{0}^t |\sigma(X_{\underline{s}}^{(\delta)})|^2 \,ds \right)^{p/2} \right]\Bigg). \notag
\end{align*}
By the linear growth condition on $\sigma$, H\"older's inequality, and \eqref{SpSz:eq8}, there exists $c_9 \in (0,\infty)$ such that  
\begin{equation} \label{SpSz:eq10}
\mathbb{E} [ \sup_{s \in [0,t]}|X_s^{(\delta)}|^p] \leq  c_9 \Bigg( 1 + \int_{0}^t \sup_{u \in [0,s]}  \mathbb{E} [|X_{\underline{u}}^{(\delta)}|^p] \,ds \Bigg) < \infty.
\end{equation}
This ensures
\begin{equation}
\mathbb{E} [ \sup_{s \in [0,t]}|X_s^{(\delta)}|^p ] \leq c_{9}\Bigg( 1 + \int_{0}^t \mathbb{E} [\sup_{u \in [0,s]}|X_{u}^{(\delta)}|^p] \,ds\Bigg).
\end{equation}
With this and \eqref{SpSz:eq10}, Gronwall's lemma yields for all $\delta\in(0,1)$ sufficiently small, 
\begin{equation} \label{SpSz:final estimate}
\mathbb{E} [ \sup_{t \in [0,T]} |X^{(\delta)}_t  |^p ] \leq  C_p^{(\text{M})}.
\end{equation}  
For the second statement by standard proof steps we obtain that there exists $c_{10} \in (0,\infty)$ such that
\begin{align} \label{SpSz:3_2}
 \left(    \mathbb{E}\Bigg[ \sup_{s \in [t, t + \delta]}  \Bigg| \int_{t}^{s} \frac{\mu( X^{(\delta)}_{\underline{u}})}{1 + \delta |\mu( X^{(\delta)}_{\underline{u}})|} \, du \Bigg|^p  \Bigg] \right)^{1/p} & \leq c_{10} \delta
\end{align} 
and hence by choosing $C_p^{(M)}$ appropriately,
\begin{align*} \label{SpSz:5_2}
&\left(\mathbb{E}\big[\sup_{s \in [t, t + \delta]} | X^{(\delta)}_s - X^{(\delta)}_{t}|^{p}\big] \right)^{1/p} \leq C_p^{(M)} \delta^{1/2}. \tag*{$\square$} 
\end{align*}
\end{proof}
\subsubsection{Estimation of the occupation time and the discontinuity crossing probabilities of the tamed-Euler--Maruyama process}
For all $x\in\mathbb{R}$ denote by $X^{(\delta),x}$ the solution of the time-continuous tamed-Euler--Maruyama scheme \eqref{SpSz:tamedeulerapprox} starting at $X^{(\delta),x}_0=x$.
In this subsection we follow the steps of \cite{SpSz:MGYa2020} and adapt all the results to our setting.
\begin{lemma} \label{SpSz:lemmaonSinitialvalue} 
Let Assumption 1 hold. Then there exists $ C^{(\text{I})}\in (0,\infty)$ such that for all $p\in[2,\infty)$, $x\in\mathbb{R}$, $t\in [0,T]$, $\delta \in(0,1)$ sufficiently small, 
\begin{equation} \label{SpSz:mom_est_euler_x}
	\Big(\mathbb{E}\Big[\sup\limits_{t \in [0, T]}|X^{(\delta),x}_t|^p\Big]\Big)^{\!1/p}\leq C^{(\text{I})}(1+|x|),\nonumber
\end{equation}
\begin{equation}  \label{SpSz:mean_sqrt_reg_x}
	\big(\mathbb{E}\big[ \sup_{s \in [t, t+\delta]}|X^{(\delta),x}_{s}-X^{(\delta),x}_{t}|^p\big]\big)^{\!1/p}\leq C^{(\text{I})}(1+|x|^{\tilde{c}+1}) \delta^{1/2}.\nonumber
\end{equation}
\end{lemma}
The proof is very similar to the proof of Lemma \ref{SpSz:lemmaonS} and therefore omitted.
\begin{lemma}\label{SpSz:occup_time2}
Let Assumption 1 hold and let $\tilde{c}\in(0,\infty)$. There exists $C^{(O)}\in (0,\infty)$ such that for all $k\in\{1,\dots,m\}$, $x\in\mathbb{R}$,
$\delta\in(0,1)$ sufficiently small, $\varepsilon \in (0,\infty)$ it holds
	\begin{equation}
		\int\limits_0^T\mathbb{P}(|X^{(\delta),x}_t-\zeta_k|\leq\varepsilon)dt\leq C^{(O)}(1+|x|^{2\tilde{c}+2})(\varepsilon+\delta^{1/2}).\nonumber
	\end{equation}
\end{lemma}
\begin{proof} 
The steps in this proof run along the same lines as in \cite{SpSz:MGYa2020}.
We obtain by Lemma \ref{SpSz:help1} and Lemma \ref{SpSz:lemmaonSinitialvalue},
\begin{equation} \label{SpSz:helpeqn2_3}
\begin{aligned}
&\mathbb{E} \left[\int\limits_0^t  \Bigg| \frac{\mu(X^{(\delta),x}_{\underline{s}})}{1 + \delta |\mu(X^{(\delta),x}_{\underline{s}})|}\Bigg| \mathbb{1}_{[\zeta_1, \zeta_m ]^c}(X^{(\delta),x}_{\underline{s}}) \,ds\right]  \\
& \leq \int\limits_0^t c (1+\mathbb{E} [\sup\limits_{t\in[0,T]}|X^{(\delta),x}_{t}|^{\tilde{c}+1}]) \,ds \leq c_1(1+|x|^{\tilde{c}+1}).
\end{aligned}
\end{equation}
For the term on $[\zeta_1, \zeta_m]$ our estimate works as in \cite{SpSz:MGYa2020} and gives 
\begin{equation} \label{SpSz:helpeqn2_4}
\begin{aligned}
&\mathbb{E} \left[\int\limits_0^t  \Bigg| \frac{\mu(X^{(\delta),x}_{\underline{s}})}{1 + \delta |\mu(X^{(\delta),x}_{\underline{s}})|}\Bigg| \mathbb{1}_{[\zeta_1, \zeta_m ]}(X^{(\delta),x}_{\underline{s}}) \,ds\right] \leq c_2(1+|x|).
\end{aligned}
\end{equation}
From here on the proof runs again along the same lines as in \cite{SpSz:MGYa2020}, but with $|x|^{\tilde{c}+1}$ instead of $|x|$.
\end{proof}
Now, let for all $k\in\{1,\dots,m\}$, $t\in[0,T]$,
$
\mathcal{Z}_k^t=\{\omega\in\Omega\colon(X^{(\delta)}_{\underline{t}}(\omega)-\zeta_k)(X^{(\delta)}_t(\omega)-\zeta_k)\le0\}.
$
\begin{lemma}\label{SpSz:help0-cross_2}
Let Assumption 1 hold. Let $s,t\in[0,T]$ with ${\underline{t}}-s\ge \delta$.
There exists a constant $C_1\in(0,\infty)$ such that for all $k\in\{1,\dots,m\}$, $\delta\in(0,1)$ sufficiently small,
\begin{equation*}
\begin{aligned}
 \mathbb{P}(\mathcal{Z}_k^s\cap \mathcal{Z}_k^t) 
& \le
 C_1\mathbb{P}(\mathcal{Z}_k^s)\delta \\&\quad + C_1 \cdot \int_\mathbb{R} \mathbb{P}\!\left(\mathcal{Z}_k^s\cap  \left\{|X^{(\delta)}_{{\underline{t}}-(t-{\underline{t}})}-\zeta_k| \le C_1 \delta^{1/2}(1+|z|)\right\}\right) \cdot e^{-\frac{z^2}{2}} dz.
\end{aligned}
\end{equation*}
\end{lemma}
\begin{proof}
The proof works as in \cite[Lemma 5]{SpSz:MGYa2020}. The only difference lies in showing the inclusion \cite[(20)]{SpSz:MGYa2020}, which is a part of \cite[Proof of Lemma 5]{SpSz:MGYa2020}, that is 
\begin{equation}\label{SpSz:star2_1}
\begin{aligned}
&\mathcal{Z}_k^t \cap \bigg\{\max_{i\in\{1,2,3\}} |\bar W_i|   \le \sqrt{2\log(T/\delta)}\bigg\}\\
&\quad  \subseteq
\bigg\{|X^{(\delta)}_{{\underline{t}}-(t-{\underline{t}})}-\zeta_k| \le \frac{12c_2(1+|\zeta_k|)\cdot(1+|\bar W_1|+|\bar W_2|)}{\sqrt{T/\delta}}\bigg\},
\end{aligned}
\end{equation}
where
\begin{equation*}
\bar W_1=\frac{W_t-W_{\underline{t}}}{\sqrt{t-{\underline{t}}}}, \qquad \bar W_2=\frac{W_{\underline{t}}-W_{{\underline{t}}-(t-{\underline{t}})}}{\sqrt{t-{\underline{t}}}}, \qquad \bar W_3=\frac{W_{{\underline{t}}-(t-{\underline{t}})}-W_{{\underline{t}}-\delta}}{\sqrt{\delta-(t-{\underline{t}})}}.
\end{equation*}
Choose $\omega \in \Omega$ such that
\begin{equation}\label{SpSz:central3}
(X^{(\delta)}_{\underline{t}}(\omega)-\zeta_k)(X^{(\delta)}_t(\omega)-\zeta_k)\le0 ~\text{\quad and \quad} \max_{i\in\{1,2,3\}}|\bar{W}_i(\omega)| \le \sqrt{2\log(T/\delta)}. \nonumber
\end{equation}
Using the linear growth condition on $\sigma$, Lemma \ref{SpSz:helpingstuff}, and the fact that for all $a,b\in\mathbb{R}$,
\begin{equation} \label{SpSz:JJJ}
\begin{aligned}
1+|a| & = 1 + |a - b + b| \leq 1 + |a-b| + |b| \leq 1 + |a-b| + |b| + |b| \cdot |a-b| \\
&= (1+|a-b|)\cdot (1+|b|),
\end{aligned}
\end{equation}
we obtain that there exist $c_1,c_2 \in(0,\infty)$ with
\begin{equation} \label{SpSz:LLL}
\begin{aligned}
&|X^{(\delta)}_{\underline{t}}(\omega)-\zeta_k|  \le |(X^{(\delta)}_{\underline{t}}(\omega)-\zeta_k) - (X^{(\delta)}_t(\omega)-\zeta_k)| \\
&  \leq  \frac{|\mu(X^{(\delta)}_{\underline{t}}(\omega))|\delta}{1 + \delta |\mu(X^{(\delta)}_{\underline{t}}(\omega))|}   + \big| \sigma(X^{(\delta)}_{\underline{t}}(\omega))\big|\cdot \big|W_t(\omega) - W_{{\underline{t}}}(\omega) \big| \\
& \le c_2(1+|X^{(\delta)}_{\underline{t}}(\omega)-\zeta_k|)\cdot(1+|\zeta_k|)\cdot \sqrt{\delta}(1 + |\bar{W}_1(\omega)| ).
\end{aligned}
\end{equation}
This shows for $\delta \in (0,1)$ sufficiently small,
\begin{align}\label{SpSz:central4}
|X^{(\delta)}_{\underline{t}}(\omega)-\zeta_k|  \leq \frac{c_2\cdot(1+|\zeta_k|)\cdot \sqrt{\delta}(1 + |\bar{W}_1(\omega)| )}{1- c_2\cdot(1+|\zeta_k|)\cdot \sqrt{\delta}(1 + |\bar{W}_1(\omega)| )}.
\end{align}
For $\delta \in (0,1)$ sufficiently small we have 
\[
c_2\cdot(1+|\zeta_k|)\cdot \sqrt{\delta}(1 + |\bar{W}_1(\omega)| ) \le  c_2\cdot (1+|\zeta_k|)\cdot \sqrt{\delta} \cdot \left(1+ \sqrt{2 \log(T/\delta)}\right) \leq \frac{1}{2}.
\]
This and \eqref{SpSz:central4} show
\begin{equation} \label{SpSz:central5_1}
\begin{aligned}
|X^{(\delta)}_{\underline{t}}(\omega)-\zeta_k| &  \leq \frac{c_2\cdot(1+|\zeta_k|)\cdot \sqrt{\delta}(1 + |\bar{W}_1(\omega)| )}{1- \frac{1}{2}} \\& = 2c_2\cdot(1+|\zeta_k|)\cdot \sqrt{\delta}(1 + |\bar{W}_1(\omega)| ).
\end{aligned}
\end{equation}
Similar to \eqref{SpSz:LLL}, again using Lemma \ref{SpSz:helpingstuff}, we obtain the estimates
\begin{equation}\label{SpSz:central9}
\begin{aligned}
|X^{(\delta)}_{{\underline{t}} - (t - {\underline{t}})}(\omega)-\zeta_k| & \leq \frac{6c_2\cdot(1+|\zeta_k|)\cdot \sqrt{\delta} \cdot (1+|\bar{W}_1(\omega)|+|\bar{W}_2(\omega)|)}{1- 6c_2\cdot(1+|\zeta_k|)\cdot \sqrt{\delta} \cdot (1+|\bar{W}_1(\omega)|+|\bar{W}_2(\omega)|)},\nonumber
\end{aligned}
\end{equation}
and
\begin{align*}
|X^{(\delta)}_{{\underline{t}} - (t - {\underline{t}})}(\omega)-\zeta_k| &\leq \frac{6c_2\cdot(1+|\zeta_k|)\cdot \sqrt{\delta} \cdot (1+|\bar{W}_1(\omega)|+|\bar{W}_2(\omega)|)}{1- \frac{1}{2}} \\
& = 12c_2\cdot(1+|\zeta_k|)\cdot \sqrt{\delta} \cdot (1+|\bar{W}_1(\omega)|+|\bar{W}_2(\omega)|). 
\end{align*}
\end{proof}
\begin{lemma}\label{SpSz:help1-cross_2}
Let Assumption 1 hold. Let $s\in[0,T)$, and let $\tilde{c}\in(0,\infty)$.
There exists a constant ${C_2}\in(0,\infty)$ such that for all $k\in\{1,\dots,m\}$, $\delta\in(0,1)$ sufficiently small,
\begin{equation*}
 \int_s^T  \mathbb{P}(\mathcal{Z}_k^s\cap \mathcal{Z}_k^t) dt
\le
{C_2}\delta^{1/2}\cdot \left(\mathbb{P}(\mathcal{Z}_k^s)
+  \mathbb{E} \!\left[\mathbb{1}_{\mathcal{Z}_k^s} \cdot (X^{(\delta)}_{{\underline{s}}+\delta}-\zeta_k)^{2\tilde{c}+2} \right]\right).
\end{equation*}
\end{lemma}
The proof runs along the same lines as \cite[Proof of Lemma 6]{SpSz:MGYa2020} or \cite[Proof of Lemma 4.4.]{SpSz:PrSz}. The only difference is that we apply Lemma \ref{SpSz:occup_time2} instead of \cite[Lemma 4]{SpSz:MGYa2020} or \cite[Lemma 4.2.]{SpSz:PrSz}, which changes one term in the statement of the Lemma to $(X^{(\delta)}_{{\underline{s}}+\delta}-\zeta_k)^{2\tilde{c}+2}$.
\begin{lemma}\label{SpSz:help3-cross_2}
Let Assumption 1 hold.
There exists a constant ${C_3}\in(0,\infty)$ such that for all $k\in\{1,\dots,m\}$, $\tilde{c}\in (0,\infty)$,
\begin{equation*}
 \int_0^T \mathbb{E} \!\left[\mathbb{1}_{\mathcal{Z}_k^s} \cdot (X^{(\delta)}_{{\underline{s}}+\delta}-\zeta_k)^{2\tilde{c}+2} \right]\, ds\le {C_3} \delta.
\end{equation*}
\end{lemma}
\begin{proof}
First, note that
\begin{equation*}
\begin{aligned}
\mathbb{1}_{\mathcal{Z}_k^s} \cdot |X^{(\delta)}_{{\underline{s}}+\delta}-\zeta_k| 
&\le 
\mathbb{1}_{\mathcal{Z}_k^s} \cdot (|X^{(\delta)}_{{\underline{s}}+\delta}-X^{(\delta)}_s| + |X^{(\delta)}_s-X^{(\delta)}_{\underline{s}}| )
.
\end{aligned}
\end{equation*}
With this and by Lemma \ref{SpSz:lemmaonS} we get that
\begin{align*}
& \int_0^T \mathbb{E} \!\left[\mathbb{1}_{\mathcal{Z}_k^s} \cdot (X^{(\delta)}_{{\underline{s}}+\delta}-\zeta_k)^{2\tilde{c}+2} \right]\, ds 
   \\&\quad\le
 2^{2\tilde{c}+1}\cdot  \sum_{k=0}^{N-1}\int_{t_k}^{t_{k+1}}\left( \mathbb{E} \!\left[|X^{(\delta)}_{{\underline{s}}+\delta}-X^{(\delta)}_s|^{2\tilde{c}+2}\right] + \mathbb{E} \!\left[|X^{(\delta)}_s-X^{(\delta)}_{\underline{s}}|^{2\tilde{c}+2}  \right]\right) \, ds 
 \\&\quad\le 2^{2\tilde{c}+1} \cdot (2 C_{2\tilde{c}+2} ^{(\text{M})} \cdot \delta) \cdot \int_0^T 1 \, ds = 2^{2\tilde{c}+1} 2 T C_{2\tilde{c}+2}^{(M)} \cdot \delta. 
\end{align*}
\end{proof}
\begin{proposition} \label{SpSz:prob-cross_2} Let Assumption 1 hold. There exists a constant $C^{(\text{cross})}\in(0,\infty)$ such that for all $k\in\{1,\dots,m\}$, $\tilde{c}\in (0,\infty)$, $\delta\in(0,1)$ sufficiently small,
\begin{equation*}
\left(\mathbb{E} \!\left[\left|\int_0^{T}   \mathbb{1}_{\{(x,y)\in\mathbb{R}^2 \colon (x-\zeta_k)(y-\zeta_k)\le0\}}(X^{(\delta)}_{\underline{s}},X^{(\delta)}_s) ds\right|^2\right]\right)^{1/2} \le C^{(\text{cross})} \cdot \delta.
\end{equation*}
\end{proposition}
The proof works exactly as the one of \cite[Proposition 1]{SpSz:MGYa2020} or \cite[Proposition 4.6]{SpSz:PrSz}.
\subsection{Proof of Theorem \ref{SpSz:mainmain}} 
As already shown in Section \ref{SpSz:first_part}, we have to estimate the two terms in \eqref{SpSz:est-triangle1}.
To estimate the first part of \eqref{SpSz:est-triangle1}, we may apply \cite[Theorem 1.1]{SpSz:HuJeKl2012}, see Section 1 herein, i.e.~there exists a constant $c_0\in(0,\infty)$ such that
\begin{align}\label{SpSz:est-euler}
 \mathbb{E}\bigg[\sup_{t\in[0,T]}|Z_t- Z^{(\delta)}_t|^2\bigg]^{1/2} \le c_0\delta^{1/2}.
\end{align}
Hence our task is to estimate the second term in \eqref{SpSz:est-triangle1}.
For all $\tau\in[0,T]$, let
\begin{equation} \label{SpSz:u}
u(\tau):=\mathbb{E}\bigg[\sup_{t\in[0,\tau]}|Z^{(\delta)}_t- G(X^{(\delta)}_t)|^2\bigg]
\end{equation}
and let for all $x_1,x_2\in\mathbb{R}$, 
$$\nu(x_1,x_2)=G'(x_1)\frac{\mu(x_2)}{1+\delta |\mu(x_2)|}+\frac{1}{2} G''(x_1) \sigma(x_2)^2.$$
By It\^o's formula we obtain
\begin{align*}
G(X^{(\delta)}_{t})
&= G(X^{(\delta)}_{0}) + \int_0^{t}\nu(X^{(\delta)}_s,X^{(\delta)}_{\underline{s}}) \,ds+\int_0^{t}G'(X^{(\delta)}_s)\sigma(X^{(\delta)}_{\underline{s}})\,dW_s,
\end{align*}
and
\begin{align} \label{SpSz:8E-i}
|Z^{(\delta)}_\tau- G(X^{(\delta)}_\tau)|^2 = \sum_{i=1}^8 E_{i,\tau},
\end{align}
where for all $s \in [0,T]$,
\begin{align*}
& E_{1,s} = -  2 \int_{0}^s  \Bigg( \frac{\delta \tilde{\mu}(Z^{(\delta)}_{\underline{t}}) |\tilde{\mu}(Z^{(\delta)}_{\underline{t}})|}{1+ \delta |\tilde{\mu}(Z^{(\delta)}_{\underline{t}})|}  \Bigg)(Z^{(\delta)}_t - G(X^{(\delta)}_t)) \, dt  \\
& E_{2,s} = 2 \int_{0}^s  \Bigg( (G'(X^{(\delta)}_{\underline{t}}) - G'(X^{(\delta)}_t))\frac{\mu(X^{(\delta)}_{\underline{t}})}{1+ \delta |\mu(X^{(\delta)}_{\underline{t}})|}  \Bigg)(Z^{(\delta)}_t - G(X^{(\delta)}_t)) \, dt  \\
& E_{3,s} =   2 \int_{0}^s \Bigg(G'(X^{(\delta)}_{\underline{t}})\left( \frac{\delta \mu(X^{(\delta)}_{\underline{t}})|\mu(X^{(\delta)}_{\underline{t}})|}{1+ \delta |\mu(X^{(\delta)}_{\underline{t}})|}\right) \Bigg)(Z^{(\delta)}_t - G(X^{(\delta)}_t)) \, dt \\
& E_{4,s} = 2 \int_{0}^s \Bigg(\tilde{\sigma}(Z^{(\delta)}_{\underline{t}}) - G'(X^{(\delta)}_t)\sigma(X^{(\delta)}_{\underline{t}}) \Bigg)(Z^{(\delta)}_t - G(X^{(\delta)}_t)) \, dW_t \\
& E_{5,s} =  \int_{0}^s \Big(\tilde{\sigma}(Z^{(\delta)}_{\underline{t}}) - G'(X^{(\delta)}_t)\sigma(X^{(\delta)}_{\underline{t}}) \Big)^2 \, dt \\
& E_{6,s} =  \int_{0}^s  \left( G''(X^{(\delta)}_{\underline{t}}) \sigma^2(X^{(\delta)}_{\underline{t}}) -        G''(X^{(\delta)}_t) \sigma^2(X^{(\delta)}_{\underline{t}}) \right)(Z^{(\delta)}_t - G(X^{(\delta)}_t)) \, dt \\
& E_{7,s} = 2 \int_{0}^s \Big((\tilde{\mu}(Z^{(\delta)}_{\underline{t}}) - \tilde{\mu}(G(X^{(\delta)}_{\underline{t}}))-(\tilde{\mu}(Z^{(\delta)}_t) - \tilde{\mu}(G(X^{(\delta)}_t))\Big)\\&\qquad \cdot (Z^{(\delta)}_t - G(X^{(\delta)}_t)) \, dt  \\
& E_{8,s} = 2 \int_{0}^s (\tilde{\mu}(Z^{(\delta)}_t) - \tilde{\mu}(G(X^{(\delta)}_t))(Z^{(\delta)}_t - G(X^{(\delta)}_t)) \, dt.
\end{align*}
The fact that $ab \leq \frac{a^2}{2} + \frac{b^2}{2}$ for all $a,b \in \mathbb{R}$ give
\begin{align*}
&\mathbb{E}\bigg[\sup_{t\in[0,\tau]}|E_{1,t}| \bigg] \leq  \delta  \int_{0}^\tau  \mathbb{E} \left[  |\tilde{\mu}(Z^{(\delta)}_{\underline{s}}) |^4\right] \,ds + \int_{0}^\tau u(s) \,ds, \\
& \mathbb{E}\bigg[\sup_{t\in[0,\tau]} |E_{3,t}|  \bigg] \leq \|G'\|_{\infty}^2 \delta  \int_{0}^\tau \mathbb{E} \left[  |\tilde{\mu}(Z^{(\delta)}_{\underline{s}}) |^4\right]\,ds + \|G'\|_{\infty}^2 \int_{0}^\tau u(s) \,ds.
\end{align*}
By \cite[Lemma 3.10.]{SpSz:HuJeKl2012} there exists $c_1 \in (0,\infty)$ such that
\begin{equation}\label{SpSz:E1E3}
\begin{aligned} 
&\mathbb{E}\bigg[\sup_{t\in[0,\tau]} |E_{1,t}|  \bigg] + \mathbb{E}\bigg[\sup_{t\in[0,\tau]} |E_{3,t}|  \bigg] \leq  c_1 \left( \delta + \int_{0}^\tau u(t) \,dt \right).\\
\end{aligned}
\end{equation}
For $E_{5,\tau}$ it holds by the estimate $(a+b)^2 \leq 2a^2 + 2b^2$ that
\begin{equation}\label{SpSz:E5_1}
\begin{aligned}
\mathbb{E}\bigg[\sup_{t\in[0,\tau]} |E_{5,t}|  \bigg] 
 &\quad\leq  2  \int_{0}^\tau \bigg( \mathbb{E}\left[|\tilde{\sigma}(Z^{(\delta)}_{\underline{s}}) -G'(X^{(\delta)}_{\underline{s}})\sigma(X^{(\delta)}_{\underline{s}})|^2\right] \\
 &\qquad+ \mathbb{E}\left[| G'(X^{(\delta)}_{\underline{s}})\sigma(X^{(\delta)}_{\underline{s}})  - G'(X^{(\delta)}_s)\sigma(X^{(\delta)}_{\underline{s}}) |^2 \right]\bigg)\, ds .
 \end{aligned}
\end{equation}
By using the Lipschitz continuity of $\tilde{\sigma}$ and the fact that $\tilde\sigma(G(X^{(\delta)}_{\underline{s}}))\\ =  G'(X^{(\delta)}_{\underline{s}})\sigma(X^{(\delta)}_{\underline{s}})$, we get for the first term in \eqref{SpSz:E5_1} that there exists $c_2 \in (0,\infty)$ such that
\begin{equation} \label{SpSz:E5_2}
\begin{aligned} 
&\int_0^\tau \mathbb{E}\left[|\tilde{\sigma}(Z^{(\delta)}_{\underline{s}}) -G'(X^{(\delta)}_{\underline{s}})\sigma(X^{(\delta)}_{\underline{s}})|^2\right] \,ds \leq c_2 \int_0^\tau u(s) \,ds.
\end{aligned}
\end{equation}
The linear growth of $\sigma$, the Lipschitz continuity of $G'$, and the fact that $X^{(\delta)}_{\underline{s}}$ is $\mathbb{F}_{\underline{s}}$-measurable give for the second term in \eqref{SpSz:E5_1},
\begin{equation} \label{SpSz:E5_3}
\begin{aligned}
& \int_0^\tau \mathbb{E} \!\left[|G'(X^{(\delta)}_{\underline{s}})\sigma(X^{(\delta)}_{\underline{s}})-G'(X^{(\delta)}_s)\sigma(X^{(\delta)}_{\underline{s}}) |^2 \right] \,ds\\
& \leq (c_{\sigma})^2(L_{G'})^2 \cdot \int_0^{\tau}\mathbb{E} \!\left[\left(1+ |X^{(\delta)}_{\underline{s}}|\right)^2 \mathbb{E} \left[ |X^{(\delta)}_{\underline{s}} - X^{(\delta)}_{s}|^2 | \mathbb{F}_{\underline{s}} \right] \right] ds.
\end{aligned}
\end{equation}
Since $W_s-W_{\underline{s}}$ is independent of $\mathbb{F}_{\underline{s}}$, $X^{(\delta)}_{\underline{s}}$ is $\mathbb{F}_{\underline{s}}$-measurable and by the linear growth of $\sigma$, we get for $s\in [0,T]$,
\begin{equation} \label{SpSz:cond-exp1}
\begin{aligned}
	&\mathbb{E} \!\left[ |X^{(\delta)}_{\underline{s}}-X^{(\delta)}_s|^2 \, \big| \, \mathbb{F}_{\underline{s}} \right]\\
	&= 2 \Big| \int_{{\underline{s}}}^s \frac{\mu(X^{(\delta)}_{\underline{s}})}{1+ \delta |\mu(X^{(\delta)}_{\underline{s}})|} \, ds \Big|^2  + 2| \sigma(X^{(\delta)}_{\underline{s}})|^2 \mathbb{E} \left[ |W_{s}-W_{{\underline{s}}}|^2 \right]\\
& \leq 4 |s - {\underline{s}}|  \bigg[  \int_{{\underline{s}}}^s | \mu(X^{(\delta)}_{\underline{s}})|^2 \left( \mathbb{1}_{[\zeta_1 , \zeta_m]}(X^{(\delta)}_{\underline{s}}) + \mathbb{1}_{[\zeta_1 , \zeta_m]^c}(X^{(\delta)}_{\underline{s}})\right)  +  c_{\sigma}^2(1+|X^{(\delta)}_{\underline{s}}|^2) \bigg]. \nonumber
\end{aligned}
\end{equation}
With the linear growth of $\mu \mathbb{1}_{[\zeta_1 , \zeta_m]}$ and the estimate
$|\mu(x)\mathbb{1}_{[\zeta_1 , \zeta_m]^c}(x)|^2 \leq 2c^2\left(1+ |x|^{2(\tilde{c}+1)} \right)$ we get
\begin{equation}
\begin{aligned}\label{SpSz:cond-exp2}
\mathbb{E} \!\left[ |X^{(\delta)}_{\underline{s}}-X^{(\delta)}_s|^2 \, \big| \, \mathbb{F}_{\underline{s}} \right] &\leq  (8 c_{\mu}^2 + 4 c_{\sigma}^2)(1+|X^{(\delta)}_{\underline{s}}|^2)\delta \\
&\quad + 8\delta c^2 \left(1+ |X^{(\delta)}_{\underline{s}}|^{2(\tilde{c}+1)} \right).
\end{aligned}
\end{equation}
Combining this with \eqref{SpSz:E5_3} yields for the second term in \eqref{SpSz:E5_1} that
\begin{equation} \label{SpSz:E5_4}
\begin{aligned}
&\int_0^{\tau}\mathbb{E} \!\left[|G'(X^{(\delta)}_{\underline{s}})\sigma(X^{(\delta)}_{\underline{s}})-G'(X^{(\delta)}_s)\sigma(X^{(\delta)}_{\underline{s}}) |^2  \right] \,ds\\
& \quad \leq (c_{\sigma})^2(L_{G'})^2 \cdot \int_0^{\tau}\mathbb{E} \bigg[|\left(1+ |X^{(\delta)}_{\underline{s}}|\right)^2 \bigg\lbrace \left( 1+ |X^{(\delta)}_{\underline{s}}|^2\right)(8 c_{\mu}^2 + 4 c_{\sigma}^2)\delta \\
&\qquad + 8\delta c^2\left(1+ |X^{(\delta)}_{\underline{s}}|^{2(\tilde{c}+1)} \right)\bigg\rbrace \bigg] ds.
\end{aligned}
\end{equation}
Hence, \eqref{SpSz:E5_2}, \eqref{SpSz:E5_4}, and Lemma \ref{SpSz:lemmaonS} establish that there exist $c_3, c_4 \in (0, \infty)$ such that
\begin{align} \label{SpSz:E5}
\mathbb{E} \bigg[\sup_{t\in[0,\tau]} |E_{5,t}|  \bigg]  & \leq c_3 \delta +  c_4 \int_{0}^\tau u(s) \, ds.
\end{align}
For $E_{4,\tau}$, by the Burkholder-Davis-Gundy inequality and  $a\cdot b \leq \frac{a^2}{2} + \frac{b^2}{2}$,
\begin{equation} \label{SpSz:E4_1}
\begin{aligned}
 \mathbb{E}\bigg[\sup_{t\in[0,\tau]} |E_{4,t}|  \bigg]  &\leq \int_{0}^\tau u(t) \,dt \\
 &\quad +  \mathbb{E} \left[ \int_{0}^\tau |\tilde{\sigma}(Z^{(\delta)}_{\underline{s}}) - G'(X^{(\delta)}_s)\sigma(X^{(\delta)}_{\underline{s}})|^2 \,ds \right]. 
 \end{aligned}
\end{equation}
By estimating the second term in \eqref{SpSz:E4_1} analog to \eqref{SpSz:E5_1}, this yields that there exists $c_5 \in (0,\infty)$ such that
\begin{align} \label{SpSz:E4}
\mathbb{E} \bigg[\sup_{t\in[0,\tau]} |E_{4,t}|  \bigg] \leq  \int_{0}^\tau u(t) \,dt + c_5 \delta.
\end{align}
For $E_{7,\tau}$ we have by H\"older's inequality and Lemma \ref{SpSz:help1}  that
\begin{align*}
& \mathbb{E}\bigg[\sup_{t\in[0,\tau]} |E_{7,t}|  \bigg]  \leq  2  \int_{0}^\tau \Bigg( \Big( \mathbb{E} \Big[\tilde{c}^2(1 + |Z^{(\delta)}_{\underline{s}}|^{\tilde{c}}+|Z^{(\delta)}_s|^{\tilde{c}})^2  \cdot|Z^{(\delta)}_{\underline{s}}- Z^{(\delta)}_s|^2\Big] \Big)^{1/2} \\&\qquad \cdot \left(\mathbb{E} \left[|Z^{(\delta)}_s - G(X^{(\delta)}_s)|^2 \right]\right)^{1/2} + \Big( \mathbb{E} \Big[\tilde{c}^2(1 + |G(X^{(\delta)}_{\underline{s}})|^{\tilde{c}}+|G(X^{(\delta)}_s)|^{\tilde{c}})^2 \\
& \cdot|G(X^{(\delta)}_{\underline{s}})- G(X^{(\delta)}_s)|^2\Big]  \Big)^{1/2}  \cdot \left(\mathbb{E}\left[|Z^{(\delta)}_s - G(X^{(\delta)}_s)|^2\right]\right)^{1/2} \Bigg) \,ds.
\end{align*}
By Lipschitz continuity of $G$, H\"older's inequality, the estimate $a\cdot b \leq \frac{a^2}{2} + \frac{b^2}{2}$, and Lemma \ref{SpSz:lemmaonS}, we get that there exist $c_{6}, c_7 \in (0,\infty)$ such that
\begin{equation} \label{SpSz:E7}
\begin{aligned}
& \mathbb{E}\bigg[\sup_{t\in[0,\tau]} |E_{7,t}|  \bigg]   \leq \int_{0}^\tau \left( 2c_6 \sqrt{\delta} \cdot \left(\mathbb{E} \left[|Z^{(\delta)}_s - G(X^{(\delta)}_s)|^2 \right]\right)^{1/2} \right) \,ds \\
& \quad \leq  \int_{0}^\tau \left( c_6  \delta + \frac{1}{2}\mathbb{E} \left[|Z^{(\delta)}_s - G(X^{(\delta)}_s)|^2 \right] \right) \,ds \leq c_7 \delta + \frac{1}{2} \int_{0}^\tau u(t) \,dt.
\end{aligned}
\end{equation}
By the one-sided Lipschitz continuity of $\tilde{\mu}$ we get that there exists $L_{\tilde{\mu}} \in (0,\infty)$ such that for all $t \in [0,T]$,
\begin{align*} 
E_{8,\tau}  \leq  2 L_{\tilde{\mu}} \int_{0}^\tau |Z^{(\delta)}_t - G(X^{(\delta)}_t)|^2 \,dt.
\end{align*}
Hence, there exists $c_{8} \in (0,\infty)$ such that
\begin{align} \label{SpSz:E8}
 \mathbb{E}\bigg[\sup_{t\in[0,\tau]} |E_{8,t}|  \bigg] \leq c_{8} \int_{0}^\tau u(t) \, dt.
\end{align}
For $E_{2,\tau}$ we have by the estimate $a\cdot b \leq \frac{a^2}{2} + \frac{b^2}{2}$,
\begin{equation} \label{SpSz:E2_1}
\begin{aligned} 
 \mathbb{E}\bigg[\sup_{t\in[0,\tau]} |E_{2,t}|  \bigg]   &\leq \int_{0}^{\tau}  \mathbb{E} \left[ | (G'(X^{(\delta)}_{\underline{s}}) - G'(X^{(\delta)}_s)|^2 \cdot |\mu(X^{(\delta)}_{\underline{s}}) |^2 \right] \,ds \\&\quad+ \int_{0}^{\tau}  \mathbb{E} \left[ |Z^{(\delta)}_s - G(X^{(\delta)}_s)|^2 \right] \, ds.
\end{aligned}
\end{equation}
By using the linear growth condition of $\mu$ on $[\zeta_1,\zeta_m]$, the polynomial growth condition of $\mu$ on $[\zeta_1,\zeta_m]^c$, and the fact that $X^{(\delta)}_{\underline{s}}$ is $\mathbb{F}_{\underline{s}}$-measurable, we obtain for the first term in \eqref{SpSz:E2_1},
\begin{align*}
& \int_0^{T} \mathbb{E}\bigg[|  \mu(X^{(\delta)}_{\underline{s}})|^2 | G'(X^{(\delta)}_{\underline{s}})- G'(X^{(\delta)}_s)|^2 \,ds \bigg] \\
& \leq  (L_{G'})^2 \int_0^{T} \mathbb{E}\bigg[ | \mu(X^{(\delta)}_{\underline{s}})|^2 \left( \mathbb{1}_{[\zeta_1,\zeta_m]}(X^{(\delta)}_{\underline{s}}) + \mathbb{1}_{[\zeta_1,\zeta_m]^c}(X^{(\delta)}_{\underline{s}}) \right) | X^{(\delta)}_{\underline{s}}- X^{(\delta)}_s|^2 \bigg] \,ds  \\
& = 2(L_{G'})^2 \int_0^{T}\mathbb{E}\bigg[ \left( c_{\mu}^2\left(1+|X^{(\delta)}_{\underline{s}}|^2\right) + c^2 \left( 1+ |X^{(\delta)}_{\underline{s}}|^{2(\tilde{c}+1)} \right) \right) \\&\quad \cdot \mathbb{E} \!\left[|X^{(\delta)}_{\underline{s}}-X^{(\delta)}_{s}|^2 \Big| \mathbb{F}_{\underline{s}} \right] \bigg] \, ds.
\end{align*}
Combining this with \eqref{SpSz:cond-exp2} and Lemma \ref{SpSz:lemmaonS} shows that there exists $c_{9} \in (0, \infty)$ such that
\begin{align} \label{SpSz:E2}
 \mathbb{E}\bigg[\sup_{t\in[0,\tau]} |E_{2,t}|  \bigg]\leq c_{9} \delta + \int_{0}^{\tau} u(s) \,ds.
\end{align}
For the term $E_{6,\tau}$ the inequality $a \cdot b \leq \frac{a^2}{2} + \frac{b^2}{2}$ gives
\begin{equation} \label{SpSz:E6_1}
\begin{aligned}
& \mathbb{E}\bigg[\sup_{t\in[0,\tau]} |E_{6,t}|  \bigg]  \leq \frac{1}{2}\mathbb{E}\bigg[ \sup_{r\in[0,\tau]} | Z^{(\delta)}_r - G(X^{(\delta)}_r) |^2 \bigg] \\&\quad +   \frac{1}{2}\mathbb{E}\bigg[ \Big|\int_0^{\tau} |G''(X^{(\delta)}_{\underline{s}}) \sigma^2(X^{(\delta)}_{\underline{s}}) - G''(X^{(\delta)}_s) \sigma^2(X^{(\delta)}_{\underline{s}}) | \,ds \Big|^2 \bigg].
\end{aligned}
\end{equation}
For recovering the optimal convergence rate $1/2$, it is essential that we do not apply the Cauchy-Schwarz inequality to the second term in \eqref{SpSz:E6_1}. Instead, we adopt an idea from  \cite{SpSz:MGYa2020}. That is, we define for all $k\in\{1,\dots,m\}$ the sets
\begin{equation*}
\mathcal{Z}_k=\{(x,y)\in\mathbb{R}^2 \colon (x-\zeta_k)(y-\zeta_k)\le0\},\qquad \mathcal{Z}=\bigcup_{k=1}^m \mathcal{Z}_k.
\end{equation*}
$\mathcal{Z}$ is the set of all $x,y\in\mathbb{R}$ which lie on different sides of a point of discontinuity of the drift.
So for $(X^{(\delta)}_{\underline{s}},X^{(\delta)}_s)\in\mathcal{Z}^c$ we have that a Lipschitz condition is satisfied by $G''$. With this,
\begin{equation}\label{SpSz:E6}
\begin{aligned}
 &\mathbb{E}\bigg[\Big|\int_0^{\tau} |G''(X^{(\delta)}_{\underline{s}})\sigma^2(X^{(\delta)}_{\underline{s}})-G''(X^{(\delta)}_s)\sigma^2(X^{(\delta)}_{\underline{s}})| ds \Big|^2 \bigg]  \leq F_1 + F_2
\end{aligned}
\end{equation}
with
\begin{equation} \label{SpSz:E6-2-1}
\begin{aligned}
F_1&= \mathbb{E}\bigg[\Big|\int_0^{\tau} |G''(X^{(\delta)}_{\underline{s}})\sigma^2(X^{(\delta)}_{\underline{s}})-G''(X^{(\delta)}_s)\sigma^2(X^{(\delta)}_{\underline{s}})| \\&\quad \cdot \mathbb{1}_{\mathcal{Z}} (X^{(\delta)}_{\underline{s}},X^{(\delta)}_s)ds\Big|^2 \bigg] ,  \\
F_2&= T \cdot \mathbb{E}\bigg[\int_0^{T} |G''(X^{(\delta)}_{\underline{s}})\sigma^2(X^{(\delta)}_{\underline{s}})-G''(X^{(\delta)}_s)\sigma^2(X^{(\delta)}_{\underline{s}}) |^2 \\&\quad \cdot \mathbb{1}_{\mathcal{Z}^c}(X^{(\delta)}_{\underline{s}},X^{(\delta)}_s) ds \bigg].
\end{aligned}
\end{equation}
For $F_2$ note that by \eqref{SpSz:cond-exp2}, Lemma \ref{SpSz:lemmaonS}, and the Lipschitz continuity of $G''$ on $\mathcal{Z}^c$, there exists $c_{10} \in (0, \infty)$ such that
\begin{equation}\label{SpSz:F4}
\begin{aligned}
&F_2  \leq 
  8T(c_\sigma)^4 (L_{G''})^2 \int_0^{T} \mathbb{E}\Big[(1+|X^{(\delta)}_{\underline{s}}|^4) \big( 8 c_{\mu}^2(1+|X^{(\delta)}_{\underline{s}}|^2)\delta \\
  &\qquad + 8\delta c^2 (1+ |X^{(\delta)}_{\underline{s}}|^{2(\tilde{c}+1)} ) + 4 c_{\sigma}^2(1+|X^{(\delta)}_{\underline{s}}|^2)\delta \Big) \Big] ds \leq c_{10} \delta.
 \end{aligned}
\end{equation}
For the term $F_1$,  observe, as in \cite{SpSz:PrSz}, that for all $k\in\{1,\dots,m\}$, $x,y \in \mathbb{R}$,
\begin{align*} 
&|x| \mathbb{1}_{\mathcal{Z}_k}(x,y)= (|\zeta_k+x-\zeta_k|) \mathbb{1}_{\mathcal{Z}_k}(x,y)  \le (|\zeta_k|+|x-\zeta_k|) \mathbb{1}_{\mathcal{Z}_k}(x,y) \\& \le (|\zeta_k|+|x-y|) \mathbb{1}_{\mathcal{Z}_k}(x,y).
\end{align*}
As in \cite[Proof of Theorem 4.7.]{SpSz:PrSz} we have that
\begin{align*} 
 (1+x^2) \cdot \mathbb{1}_{\mathcal{Z}}(x,y) &=  (1+x^2) \cdot \mathbb{1}_{\cup_{k=1}^{m} \mathcal{Z}_k}(x,y) \\
 &\le  2m|x-y|^2 + (1+2\max\{|\zeta_1|,|\zeta_m|\}^2) \cdot \sum_{k=1}^m  \mathbb{1}_{\mathcal{Z}_k}(x,y).
\end{align*}
This yields
\begin{align*}
&F_1  \le
  16 (c_\sigma)^4 \|G''\|_\infty ^2\cdot\mathbb{E}\bigg[\Big|\int_0^{\tau} (1+|X^{(\delta)}_{\underline{s}}|^2)\cdot \mathbb{1}_{\mathcal{Z}}  (X^{(\delta)}_{\underline{s}},X^{(\delta)}_s)ds\Big|^2 \bigg]
    \\&\le
    16 (c_\sigma)^4 \|G''\|_\infty ^2\cdot\mathbb{E}\bigg[\Big|\int_0^{T} \Big(2m|X^{(\delta)}_{\underline{s}}-X^{(\delta)}_s|^2 
    \\&\quad+ (1+2\max\{|\zeta_1|,|\zeta_m|\}^2) \cdot \sum_{k=1}^m  \mathbb{1}_{\mathcal{Z}_k}  (X^{(\delta)}_{\underline{s}},X^{(\delta)}_s)\Big)ds\Big|^2 \bigg]  
     \\& \le
    128mT (c_\sigma)^4 \|G''\|_\infty ^2 m^2 \cdot\mathbb{E}\bigg[\int_0^{T} |X^{(\delta)}_{\underline{s}}-X^{(\delta)}_s|^4 ds \bigg] 
    \\&\quad+ 32 (c_\sigma)^4 \|G''\|_\infty ^2 (1+2\max\{|\zeta_1|,|\zeta_m|\}^2)^2 \cdot \mathbb{E}\bigg[\Big|\int_0^{T} \sum_{k=1}^m  \mathbb{1}_{\mathcal{Z}_k} (X^{(\delta)}_{\underline{s}},X^{(\delta)}_s)\,ds\Big|^2 \bigg].
\end{align*}
Together with Lemma \ref{SpSz:lemmaonS} we therefore have   
\begin{equation} \label{SpSz:F1}
\begin{aligned}
&F_1 \leq 
    128 (c_\sigma)^4 \|G''\|_\infty ^2 m^2 T^2 C^{(M)} \delta
    \\&\qquad+ 32(c_\sigma)^4 \|G''\|_\infty ^2 m^2 (1+2\max\{|\zeta_1|,|\zeta_m|\}^2)^2 \\& \qquad \cdot \mathbb{E}\bigg[\Big|\int_0^{T} \mathbb{1}_{\mathcal{Z}}  (X^{(\delta)}_{\underline{s}},X^{(\delta)}_s)\,ds\Big|^2 \bigg].
    \end{aligned}
\end{equation}
Proposition \ref{SpSz:prob-cross_2} yields 
\begin{equation}\label{SpSz:F1-2}
\begin{aligned}
 \mathbb{E}\bigg[\Big|\int_0^{T}  \mathbb{1}_{\mathcal{Z}_k}  (X^{(\delta)}_{\underline{s}},X^{(\delta)}_s) \,ds\Big|^2\bigg]
\le C^{(\text{cross})}\cdot \delta .
\end{aligned}
\end{equation} 
Combining \eqref{SpSz:E6_1}, \eqref{SpSz:E6}, \eqref{SpSz:E6-2-1}, \eqref{SpSz:F4}, \eqref{SpSz:F1}, and \eqref{SpSz:F1-2} we get that there exists $c_{11} \in (0, \infty)$ such that 
\begin{align}\label{SpSz:E6-final}
 \mathbb{E}\bigg[\sup_{t\in[0,\tau]} |E_{6,t}|  \bigg]  &\leq \frac{1}{2} \int_0^{\tau} u(t) \,dt + c_{11} \delta.
\end{align}
Combining \eqref{SpSz:u} with \eqref{SpSz:8E-i} and with the estimates \eqref{SpSz:E1E3}, \eqref{SpSz:E5}, \eqref{SpSz:E4}, \eqref{SpSz:E7}, and \eqref{SpSz:E8}, \eqref{SpSz:E2}, and \eqref{SpSz:E6-final} shows that 
there exist constants $c_{12},c_{13} \in \left(0, \infty \right)$ such that
\begin{align*}
0 \leq u(\tau) &\leq c_{12} \int_0^{\tau} u(s) \, ds + c_{13} \delta.
\end{align*}
By Lemma \ref{SpSz:lemmaonS} and the fact that $[0,T] \ni t \mapsto \mathbb{E}\left[\sup_{s\in[0,t]}|Z^{(\delta)}_s- G(X^{(\delta)}_s)|^2\right]$ is Borel measurable, Gronwall's inequality yields for all $\tau \in \left[0,T\right]$,
\begin{align*}
u(\tau)  \leq c_{13} \exp(c_{12} T) \cdot \delta.
\end{align*}
Combining this with \eqref{SpSz:est-lip}, \eqref{SpSz:est-triangle1}, and \eqref{SpSz:est-euler} finally yields
\begin{align*}
 \bigg(\mathbb{E}\bigg[\sup_{t\in[0,T]}|X_t- X^{(\delta)}_t|^2\bigg]\bigg)^{\!1/2}\le L_{G^{-1}}(c_0 \delta)^{1/2} +  L_{G^{-1}}(c_{13} \exp(c_{12}T) \cdot \delta)^{1/2}.
\end{align*}
\section*{Acknowledgements}
K. Spendier and M. Sz\"olgyenyi are supported by the Austrian Science Fund (FWF): DOC 78. The authors thank Verena Schwarz for useful discussions and proof reading and an anonymous referee for utmost carefully reading our paper and making very helpful comments. 

\bibliography{TamedEulerPaperarXiv}

\vspace{2em}
\centerline{\underline{\hspace*{16cm}}}

\noindent Michaela Sz\"olgyenyi  \\
Department of Statistics, University of Klagenfurt, Universit\"atsstra\ss{}e 65-67, 9020 Klagenfurt, Austria\\
michaela.szoelgyenyi@aau.at\\

 \noindent Kathrin Spendier \Letter  \\ 
Department of Statistics, University of Klagenfurt, Universit\"atsstra\ss{}e 65-67, 9020 Klagenfurt, Austria\\
kathrin.spendier@aau.at\\

\end{document}